\documentclass{amsart}

\usepackage{tikz}
\usepackage[all]{xy}
\usepackage[usenames,dvipsnames]{pstricks}
\usepackage{epsfig}
\usepackage{pst-grad} % For gradients
\usepackage{pst-plot} % For axes
\usepackage{ebezier}
\usepackage{latexsym}\usepackage{amsfonts}\usepackage{amssymb}\usepackage{amsthm}
\usepackage{amsmath}\usepackage{newlfont}\usepackage{graphicx}\usepackage{doc}
\usepackage{mathrsfs}
\usepackage[
        colorlinks, citecolor=red,
        backref,
        pdfauthor={JK,DM,MP},
        pdftitle={C(F)}
]{hyperref}

\newtheorem{thrm}{Theorem}

\newtheorem{cor}[thrm]{Corollary}
\newtheorem{defn}[thrm]{Definition}
\newtheorem{rem}[thrm]{Remark}
\newtheorem{prop}[thrm]{Proposition}

\def \Dj{\mbox{\raise0.3ex\hbox{-}\kern-0.4em D}}

\begin{document}

\title{Complexity function and entropy of induced maps on hyperspaces of continua}

\author{Jelena Kati\'c}
\address{Matemati\v{c}ki fakultet, Univerzitet u Beogradu}
\email{jelena.katic@matf.bg.ac.rs}

\author{Darko Milinkovi\'c}
\address{Matemati\v{c}ki fakultet, Univerzitet u Beogradu}
\email{darko.milinkovic@matf.bg.ac.rs}

\author{Milan Peri\'c}
\address{Matemati\v{c}ki fakultet, Univerzitet u Beogradu}
\email{milan.peric@matf.bg.ac.rs}

\thanks{The authors are partially supported by the Ministry of  Science, Technological Development and Innovation, Republic of Serbia, through the project 451-03-33/2026-03/200104.}

\maketitle

\begin{abstract} We use the complexity function of an invariant, not necessary closed, subset of a two-sided shift space to compute the polynomial entropy of the induced dynamics on the hyperspace of continua for certain one-dimensional dynamical systems. We also provide a simple criterion for $f$ that implies $C(f)$ has infinite topological entropy.

\end{abstract}

MS Classification: Primary 37B40, Secondary 	37B10, 54F16, 54F50

%%%%%%%%%%%%%%%%%%
%%%%%%%%%%%%%%%%%%		
%%%%%%%%%%%%%%%%%%

\section{Introduction}

Every continuous map on a compact metric space $X$ induces a continuous map $2^f$ (called the {\it induced map}) on the hyperspace $2^X$ of all nonempty closed subsets of $X$. If $X$ is connected, we consider the hyperspace $C(X)$ consisting of all nonempty closed and connected subsets of $X$, and the induced map $C(f)=2^f|_{C(X)}$.  A natural question arises: what are the possible relations between the given (individual) dynamics on $X$ and the induced (collective) dynamics on the hyperspace. Despite progress over the past few decades, this relationship remains mostly unexplored and continues to draw considerable attention. For instance,  it is known that certain dynamical properties of the system $(X,f)$ are preserved in the induced system $(2^X,2^f)$, such as Li-Yorke chaos (see~\cite{GKLOPP}) and positive topological entropy (see~\cite{LR}). On the other hand, some properties of $(2^X,2^f)$ also imply the same properties for $(X,f)$ - for example, transitivity (see~\cite{RF}). However, for some properties there is no implication in any direction. For example, neither Devaney chaos of $(X,f)$ implies Devaney chaos of $(2^X,2^f)$, nor does Devaney chaos of $(2^X,2^f)$ imply Devaney chaos of $(X,f)$ (see~\cite{GKLOPP}). Without attempting to be exhaustive, we mention just a few significant contributions in this area: Borsuk and Ulam~\cite{BU}, Bauer and Sigmund~\cite{BS}, Rom\'an-Flores~\cite{RF}, Banks~\cite{B}, Acosta, Illanes and M\'endez-Lango~\cite{AIM}.

The topological entropy $h(f)$ is a classical measure of the complexity of a dynamical system $(X,f)$, quantifying the average exponential growth in the number of distinguishable orbit segments. It is clear that the topological entropy of the induced system on a hyperspace is greater than or equal to that of $f$, since the latter is a factor of the former; see Subsection~\ref{subsec:entropy}. The topological entropy of the induced map has been studied by Bauer and Sigmund [7], Kwietniak and Oprocha~\cite{KO}, Lampart and Raith~\cite{LR}, Hern\'{a}ndez and M\'{e}ndez~\cite{HM}, Arbieto and Bohorquez~\cite{AB}, among others. It is known that if $f$ has positive topological entropy, then $2^f$ has infinite topological entropy~\cite{BS}. Regarding the topological entropy of $C(f)$, there are examples where $h(f)>0$ and $C(f)$ is both finite and infinite, see~\cite{KO} and~\cite{LR}. 

In the case when $h(f)=0$, Lampart and Raith~\cite{LR} proved that for any homeomorphism $f$ of the circle (for which $h(f)=0$ always holds), we have $h(C(f))=0$. They also proved the same result for homeomorphisms of intervals and graphs.  Arbieto and Bohorquez proved that a Morse-Smale diffeomorphism $f:M\to M$ satisfies $h(C(f))=\infty$ for $\dim M\ge 2$ and $h(C(f))=0$ for $\dim M=1$ (see~\cite{AB}).

It is known that the topological entropy is concentrated on the non-wandering set, i.e. $h(f)=h(f|_{{NW}(f)})$, see~\cite{K}. However, somewhat unexpectedly, the existence of a wandering point in a dynamical system on a manifold of dimension greater than $1$ implies that $h(C(f))=\infty$. This is the first result of the paper.

\begin{thrm}\label{thmr:infty}
Let $M$ be a topological manifold which is compact, connected and of dimension greater or equal than $2$ and $f:M\to M$ a homeomorphism that has a wandering point. Then $h(C(f))=\infty$.\qed
\end{thrm}

Since Morse--Smale diffeomorphisms have finitely many non-wandering points, as a corollary of Theorem~\ref{thmr:infty}, we obtain the apreviously mentioned result from~\cite{AB} using alternative methods.

\begin{cor}\label{cor} Let $f:M\to M$ be a Morse-Smale diffeomorphism and $\dim M\ge 2$. Then $h(C(f))=\infty$.
\qed
\end{cor}

In systems of low complexity, i.e., with zero topological entropy, one can consider {\it polynomial entropy} instead. The definitions of topological and polynomial entropy differ in that the former measures the overall {\it exponential} growth of orbit complexity, while the latter captures the growth rate on a {\it polynomial} scale. Polynomial entropy can serve as a finer tool for distinguishing among systems with zero topological entropy.

As an example, consider a rotation on the circle versus a homeomorphism with both periodic and wandering points. The first system is clearly simpler than the second; however, the topological entropy of both systems is zero. In~\cite{L}, Labrousse proved that polynomial entropy is an invariant that distinguishes between these two circle homeomorphisms.

Topological and polynomial entropy share some basic properties: they are both conjugacy invariants, depend only on the topology (not on the specific choice of metric inducing it), satisfy the finite union property, and obey a product formula. However, there are important differences as well. Certain properties such as the power formula, the $\sigma$-union property, and the variational principle hold for topological entropy but do not necessarily hold for polynomial entropy (see~\cite{M1,M2}). The following consequence of Bowen's formula (see Theorem 17 in~\cite{Bo}) also holds for topological entropy but not for polynomial entropy: if $\pi:X\to Y$ is continuous semicongugacy between dynamical systems $(X,f)$ and $(Y,g)$ and is also uniformly finite-to-one (that is, there exists $M > 0$ such that $|\pi^{-1}(y)| \le M$ for all $y \in Y$), then $h(f)=h(g)$.

As mentioned above, $h(C(f))=0$ when $f$ is a homeomorphism of a graph. Moreover, in~\cite{JB}, the author showed that for a homeomorphism $f$ of a dendrite, $h(C(f))\in\{0,\infty\}$. Therefore, we investigate the polynomial entropy of $C(f)$ on certain one-dimensional spaces.

\begin{thrm}\label{thm:star1}
Let $X=S_k$ be a star with a branch point of order $k$ and $f:X\to X$ a homeomorphism. 
If every edge possesses a wandering point, then $h_{\mathrm{pol}}(C(f))=k$.\qed
\end{thrm}

The proofs of Theorem~\ref{thmr:infty} and Theorem~\ref{thm:star1} rely on expressing the topological and polynomial entropy in terms the complexity function of certain subsets of the two-sided shift space (see~Subsetion~\ref{subsec:entr-subshift}).

The following corollary of Theorem~\ref{thm:star1} provides a complete description of $h_{\mathrm{pol}}(C(f))$ for a homeomorphism $f$ of a graph.

\begin{cor} Let $X$ be a graph and $f:X\to X$ a homeomorphism. Let $k$ be the maximal number of edges with a common branch point that possess a wandering point. Then
\begin{itemize}\label{cor:graph1}
\item $h_{\mathrm{pol}}(C(f))=k$ if $k\ge 2$;
\item $h_{\mathrm{pol}}(C(f))=2$ if $k\le 1$ and there exists a closed path $C\subset X$ such that $f^n|_{C}$ is not conjugate to the circle rotation, for all $n$;
\item $h_{\mathrm{pol}}(C(f))=0$ otherwise.
\end{itemize}
\end{cor}

Lastly, we give an affirmative answer to a question posed in~\cite{AB}.
\begin{thrm}\label{thm:question} There exists a dynamical system such that:
$$h(C(f))<h(C_n(f))<h(C_{n+1}(f))<h(2^f)$$ for all $n\ge 2$. There exists a dynamical system such that:
$$h_{\mathrm{pol}}(C(f))<h_{\mathrm{pol}}(C_n(f))<h_{\mathrm{pol}}(C_{n+1}(f))<h_{\mathrm{pol}}(2^f)$$ for all $n\ge 2$.\end{thrm}

\section{Preliminaries}

Let us recall some notions and properties that will be used in the proofs.

\subsection{Hyperspaces and induced maps}

For a compact metric space $(X,d)$, the hyperspace $2^X$ denotes the collection of all nonempty closed subsets of $X$.
The topology on $2^X$ is induced by the Hausdorff metric
$$d_H(A,B):=\inf\{\varepsilon>0\mid A\subset U_\varepsilon(B),\;B\subset U_\varepsilon(A)\},$$
where
\begin{equation}\label{eq:neighb}
U_\varepsilon(A):=\{x\in X\mid d(x,A)<\varepsilon\}.
\end{equation}
The space $2^X$, equipped with the Hausdorff metric, is called a {\it hyperspace induced by $X$}. It is compact with respect to Hausdorff metric.

We consider the following susbets of $2^X$:
\begin{itemize}
\item $C(X):=\{K\in 2^X\mid K\;\mbox{is connected}\}$;
\item $C_n(X):=\{K\in 2^X\mid K\;\mbox{has at most {\it n} connected components}\}$;
\item $F_n(f):=\{K\in 2^X\mid K\;\mbox{has at most {\it n} elements}\}$;
\item $F_\omega:=\bigcup_{n\in\mathbb{N}} F_n(X)$.
\end{itemize}

If $X$ is also connected (and hence a continuum), then $C(X)$ is compact and connected. The set $C(X)$ is called the {\it hyperspace of subcontinua} of $X$. The set $F_n(X)$ is called the \textit{$n$-fold symmetric product of} $X$. The sets $C(X)$, $F_n(X)$ and $C_n(X)$ are closed in $2^X$, but $F_\omega(X)$ is not. Moreover, the closure of $F_\omega(X)$ is the whole space $2^X$.

If $f:X\to X$ is continuous, then it induces continuous maps
$$\begin{aligned}2^f&:2^X\to 2^X,\quad 2^f(A):=\{f(x)\mid x\in A\}\\
C(f)&:C(X)\to C(X),\quad C(f):=2^f|_{C(X)}\\
C_n(f)&:C_n(X)\to C_n(X),\quad C_n(f):=2^f|_{C_n(X)}\\
F_n(f)&:F_n(X)\to F_n(X),\quad F_n(f):=2^f|_{F_n(X)}.\\
\end{aligned}$$
If $f$ is a homeomorphism, so are $2^X$, $C(f)$, $C_n(f)$ and $F_n(f)$.

The sets $C(X)$, $C_n(X)$, $F_n(X)$ and $F_\omega(X)$ are all $2^f$-invariant. 

In this paper we mostly deal with the hyperspace $C(X)$.

We will use small latin letters $x$ for points in the initial space $X$, and capital latin letters $K$ for points in the induced hyperspaces.

We denote the open and close balls by $B(x,r)$ and $B[x,r]$. %An open and a closed balls in $C(X)$ will be denoted by $B_H(X,r)$ and $B_H[X,r]$.

\subsection{Graph, $k$-od space (star) and tree}

A metric space $D$ is a \textit{dendrite} if it is nonempty, connected, locally connected and compact and contains no simple closed curves. We say that 
\begin{itemize}
\item $x\in D$ is an \textit{end point} of $D$ if $D\backslash\{x\}$ is connected
\item $x\in D$ is a \textit{cut point} of $D$ if $D\backslash\{x\}$ is not connected.
\end{itemize}
Note that if $D=\{x\}$ is a singleton, then $x$ is its end point.
 We call the number of components of $D\backslash\{x\}$ {\it the order of} $x$. If the order of $x$ is equal to $2$, then $x$ is an \textit{ordinary point} of $D$. If the order of $x$ is greater than $2$, then $x$ is a \textit{branch point} of $D$. We denote by $B(D)$ the set of all branch points and by $E(D)$ the set of all endpoints of $D$. An arc (a homeomorphic image of an interval) that connects two branch points, two end points, or a branch point and an end point is called an {\it egde}. A tree is a dendrite with a finite set of end points. A $k$-{\it star ($k$-od space}) is a tree with only one branch point of order $k$.

It is known that every subcontinuum (i.e., connected and closed subset) of a dendrite is also a dendrite; therefore, every subcontinuum of a tree is again a tree (see~\cite{N}).

By a {\it graph} we mean a compact metric space $G$ together with a finite zero-dimensional set $V\subset G$ (called {\it vertices}) such that $G\setminus V$ is homeomorphic to a finite disjoint union of intervals $(0,1)$. We call the closure of each connected component of $G\setminus V$ an {\it edge}. Note that (unlike a dendrite) a graph may contain a closed path, i.e.\ homeomorphic copy of the circle. A branch point in a graph is a vertex where three or more egdes meet.

Note that a graph without a closed path is a tree; therefore, a star is a special type of tree, and a tree is both a special type of graph and a special type of dendrite.

\subsection{Topological and polynomial entropy}\label{subsec:entropy}

Let $X$ be a compact metric space, and $f:X\rightarrow X$ a continuous map. Denote by $d_n^f(x,y)$ the dynamic metric (induced by $f$ and $d$):
$$
d_n^f(x,y)=\max\limits_{0\leq k\leq n-1}d(f^k(x),f^k(y)).
$$

Fix $Y\subseteq X$. For $\varepsilon>0$, we say that a finite set $E\subset X$ is \textit{$(f,n,\varepsilon)$-separated} (with respect to $f$) if for every two different $x,y \in E$ it holds $d_n^f(x,y)\geq \varepsilon$. Let $\mathrm{sep}(n,\varepsilon;Y)$ denote the maximal cardinality of an $(n,\varepsilon)$-separated set $E$ contained in $Y$.\label{def:sep}

\begin{defn} The \textit{topological entropy} of the map $f$ on the set $Y$ is defined by
$$
h(f;Y)=\lim\limits_{\varepsilon \rightarrow 0}\limsup\limits_{n\rightarrow \infty}\frac{\log \mathrm{sep}(n,\varepsilon;Y)}{n}.
$$
\end{defn}
We we also use the notation $\mathrm{sep}(f,n,\varepsilon;Y)$, if we want to emphasize the map in question, as well as $\mathrm{sep}(n,\varepsilon)$ when $Y=X$.
If $X=Y$ we abbreviate $h(f):=h(f;X)$. It is obvious that $h(f)\ge h(f;Y)$, for all $Y\subset X$. If $Y$ is a closed and $f$-invariant subset, then $h(f;Y)=h(f|_Y)$.

The topological entropy can also be defined via coverings with sets of $d^n_f$-diameters less than $\varepsilon$, via coverings by balls of $d_n^f$-radius less than $\varepsilon$, or via open covers. We refer the reader to~\cite{Li} and the references therein for a brief survey on the topological entropy (see also Section~\ref{sec:proof}).

The following property of the topological entropy is well known: if the dynamical system $(X_2,f_2)$ is a {\it factor} of the dynamical system $(X_1,f_1)$ (meaning that there exists a countinuous surjection $\varphi:X_1\to X_2$ satisfying $\varphi\circ f_1=f_2\circ\varphi$), then $h(f_2)\le h(f_1)$. However, we will use the relative version of it, with respect to $Y_1$ and $Y_2$, when the spaces $Y_1$ and $Y_2$ are not necessarily compact. More precisely, the following proposition holds.

\begin{prop}\label{prop:lip}
Let $(X_1,d_1)$ and $(X_2,d_2)$ be two compact metric spaces and $Y_j\subset X_j$ be an $f_j$-invariant (not necessarily compact) set and $\pi:Y_1\to Y_2$ a continuous surjection satisfying $\pi\circ f_1=f_2\circ\pi$. If $\pi$ a Lipschitz map with Lipschitz constant $L$, then
$$h(f_1;Y_1)\ge h(f_2;Y_2).$$
\end{prop}

\noindent{\it Proof.} Let $E=\{y_1,\ldots,y_N\}\subset Y_2$ is an $(f_2,n,\varepsilon)$-separated set and $x_j\in Y_1$ such that $\pi(x_j)=y_j$, for $j\in\{1,2\}$. We have
$$\begin{aligned}
&d_{X_1}(f^k_1(x_i),f_1^k(x_j))\ge\frac{1}{L} d_{X_2}(\pi(f_1^k(x_i)),\pi(f_1^k(x_j)))=\\
&\frac{1}{L}d_{X_2}(f_2^k(\pi(x_i)),f_2^k(\pi(x_j)))=\frac{1}{L}d_{X_2}(f_2^k(y_i),f_2^k(y_j)),
\end{aligned}$$ so
$$(d_{X_1})_n^{f_1}(x_i,x_j)\ge \frac{1}{L}(d_{X_2})_n^{f_2}(y_i,y_j)\ge\frac{\varepsilon}{L},$$ and therefore the set 
$$\widetilde{E}:=\{x_1,\ldots,x_N\}\subset Y_1$$ is an $(f_1,n,\varepsilon/L)$-separated set. We have:
$$\mathrm{sep}(f_1,n,\varepsilon/L;Y_1)\ge\mathrm{sep}(f_2,n,\varepsilon;Y_2),$$ so
$$h(f_1;Y_1)\ge h(f_2;Y_2)$$
\qed

\vspace{5mm}

When the topological entropy of the system is equal to zero, one useful tool for measuring compelxity is {\it polynomial entropy}. It is defined as 
$$
h_{\mathrm{pol}}(f;Y)=\lim\limits_{\varepsilon \rightarrow 0}\limsup\limits_{n\rightarrow \infty}\frac{\log \mathrm{sep}(n,\varepsilon;Y)}{\log n},
$$
i.e.\ it measures the growth rate on a polynomial scale. With respect to factors and conjugacy, polynomial and topological entropy shares the same properties. Proposition~\ref{prop:lip} also holds for polynomial entropy. We will also use the following property of polynomial entropy: if $f:X\to X$, $g:Y\to Y$ are two continuous map, and
$$f\times g:X\times Y\to X\times Y,\quad (f\times g)(x,y):=(f(x),g(y)),$$ then
$$h_{\mathrm{pol}}(f\times g)=h_{\mathrm{pol}}(f)+h_{\mathrm{pol}}(g),$$ so for $f^{\times k}:X^{\times k}\to X^{\times k}$ it holds:
$$h_{\mathrm{pol}}(f^{\times k})=kh_{\mathrm{pol}}(f).$$ 
 If $Y=\bigcup_{j=1}^mY_j$ where $Y_j$ are $f$-invariant, then 
\begin{equation}\label{eq:finite-union}
 h_{\mathrm{pol}}(f;Y)=\max\{h_{\mathrm{pol}}(f;Y_j)\mid j=1,\ldots,m\},
\end{equation} 
see~\cite{M2}. This is also true for the topological entropy.
Let us also mention a property of the polynomial entropy that does not hold for the topological entropy, with the same proof. If we denote $f^k:=f\circ\ldots\circ f$, then
$$h_{\mathrm{pol}}(f^k)=h_{\mathrm{pol}}(f).$$

\vspace{5mm}

We will also use the following notions from topological dynamics.
A point $x\in X$ is {\it wandering} if there exists a neighbourhood $U\ni x$ such that $f^n(U)\cap U=\emptyset$, for all $n\ge 1$.

A point that is not wandering is said to be {\it non-wandering}. We denote the set of all non-wandering points by $NW(f)$. The set $NW(f)$ is  closed and $f$-invariant and it holds $h(f)=h(f|_{NW(f)})$ (see~\cite{K}).

\subsection{Entropies for subshifts}\label{subsec:entr-subshift}

In general, computing the topological entropy of a dynamical system can be difficult. However, for the shift map, there is a particularly convenient method to do so. For a finite set $A$ (called an {\it alphabet}), denote by $A^\mathbb{Z}$ the set of all two-sided sequences $\mathbf{x}=\{x_n\}_{n=-\infty}^\infty$ where $x_n\in A$. The metric on $A^\mathbb{Z}$ is defined as:
$$d(x,y):=2^{-k},$$ where
$$k=\min\{|n|\mid x_n\neq y_n\}.$$ 
The topology induced by $d$ coincides Tychonoff product topology on $A^\mathbb{Z}$.
The space $A^\mathbb{Z}$ is a compact, totally disconnected metric space without isolated points. The {\it shift map} is defined as:
\begin{equation}\label{eq:shift-A}
\sigma:A^\mathbb{Z}\to A^\mathbb{Z},\quad\{\sigma(\mathbf{x})\}_n:=x_{n+1}
\end{equation}
and it is a homeomorphism. Any closed subset $\Sigma\subseteq A^\mathbb{Z}$ with $\sigma(\Sigma)=\Sigma$ is said to be a {\it two-sided subshift} or a {\it subshift}, for short.

Let $\Sigma$ be a subshift. We say that a finite sequence $(a_1,\ldots,a_m)\in A^m$ is a {\it subword} of $\mathbf{x}\in A^\mathbb{Z}$, and write $(a_1,\ldots,a_m)\sqsubseteq\mathbf{x}$, if there exists $l\in\mathbb{Z}$ such that
$$(\sigma^l(\mathbf{x})_1,\ldots,\sigma^l(\mathbf{x})_m)=(a_1,\ldots,a_m).$$
We define the {\it set of words of length} $m$ as
\begin{equation}\label{eq:L^m}
\mathscr{L}^m(\Sigma):=\{(a_1,\ldots,a_m)\sqsubseteq\mathbf{x}\mid\mathbf{x}\in\Sigma\}\subseteq A^m
\end{equation}
and the {\it complexity function} by
$$
p_\Sigma:\mathbb{N}\to\mathbb{N},\quad p_\Sigma(m):=\left|\mathscr{L}^m(\Sigma)\right|,
$$
where $|\cdot|$ denotes the cardinality of the set.
It is known that one can compute the topological entropy by using the complexity function. Namely, if $\Sigma\subset A^{\mathbb{Z}}$ is a subshift it holds:
$$h(\sigma|_\Sigma)=\lim\limits_{m\to\infty}\frac{\log p_{\Sigma}(m)}{m}$$
(see~\cite{K}).

The same formula holds for the polynomial entropy.

\begin{prop}{\bf [Proposition 3 in~\cite{DKL}.]}\label{prop:complexity-f} Let $\Sigma$ be a subshift. Then 
$$h_{\mathrm{pol}}(\sigma|_\Sigma)=\limsup\limits_{n\to\infty}\frac{\log p_\Sigma(m)}{\log m}.$$
\end{prop}

For our purposes, we need slightly more general result concerning the calculation of entropies by using the complexitiy function.

\begin{thrm}\label{thm:complexity}
Let $\Sigma\subset A^\mathbb{Z}$ be a $\sigma$-invariant set, not nessecarily closed, and $p_\Sigma$ defined as above. Then
$$h(\sigma;\Sigma)=\lim\limits_{m\to\infty}\frac{\log p_{\Sigma}(m)}{m},\quad
h_{\mathrm{pol}}(\sigma;\Sigma)=\lim\limits_{m\to\infty}\frac{\log p_{\Sigma}(m)}{\log m},$$
where $p_\Sigma(m)=\left|\mathscr{L}^m(\Sigma)\right|$, and $\mathscr{L}^m(\Sigma)$ is defined as in~(\ref{eq:L^m}).
\end{thrm}

Since the proof of Theorem~\ref{thm:complexity} is somewhat technical and follows the proof of the case when $\Sigma$ is a subshift (i.e., compact and $\sigma$-invariant), we postpone it to Section~\ref{sec:proof}.

\section{Polynomial entropy for induced hyperspace of a star}

We now restate and prove one of our main results (Theorem~\ref{thm:star1}).

\begin{thrm}\label{thrm:star}
Let $X=S_k$ be a star with $k$ edges and $f:X\to X$ a homeomorphism. 
If $$X=\bigcup_{j=1}^k I_j$$ where $I_j$ is homeomorphic to $[0,1]$ and every $I_j$ possesses a wandering point, then $h_{\mathrm{pol}}(C(f))=k$.
\end{thrm}

\noindent{\it Proof.} Suppose that $k>2$, since when $k\in\{1,2\}$, $S_k$ is homeomorphic to the interval, and this case is covered in~\cite{DK}. Let us denote by $b$ the branch point of $X$. Since $f$ is a homeomorphism, and $k>2$, $f(b)=b$ and for every $i\in\{1,\ldots,k\}$, there exists $j\in\{1,\ldots,k\}$ such that $f(I_i)=I_j$.

\vspace{1mm}
\noindent{\bf Step I.} There exists $m\in\mathbb{N}$ such that$f^m(I_j)=I_j$ for $j\in\{1,\ldots,k\}$. Indeed, since there are finitely many $I_j$'s, there exists $r>n$ such that $f^r(I_1)=f^n(I_1)$. We conclude that $I_{1}=f^{r-n}(I_1)$. Denote by 
$$X':=X\setminus(I_1\cup f(I_1)\cup\ldots\cup f^{r-n-1}(I_1))\,\cup\,\{b\},\quad f':=f^{r-n}.$$ We repeat the same argument for $f':X'\to X'$ and, after finitely many step, we get the desired $m$.

\vspace{1mm}
\noindent{\bf Step II.} Every $I_j$ has a wandering point of $f^m$. Indeed, we can identify $f^m|_{I_j}$ with a homeomorphism
$$g:[0,1]\to [0,1],\quad g(0)=0,$$ which is obviously strictly increasing map. If there exists $x\in [0,1]$ with $f(x)\neq x$, than the sequence $f^n(x)$ is strictly increasing if $f(x)>x$, and strictly decreasing if $f(x)<x$. In both cases it is easy to show that $x$ is wandering for $f^m$. If $g(x)=x$ for all $x\in[0,1]$ then $f^m|_{I_j}=\mathrm{Id}$, implying that $I_j$ does not have a wandering point for $f$, which is impossible by the assumptions of Theorem.

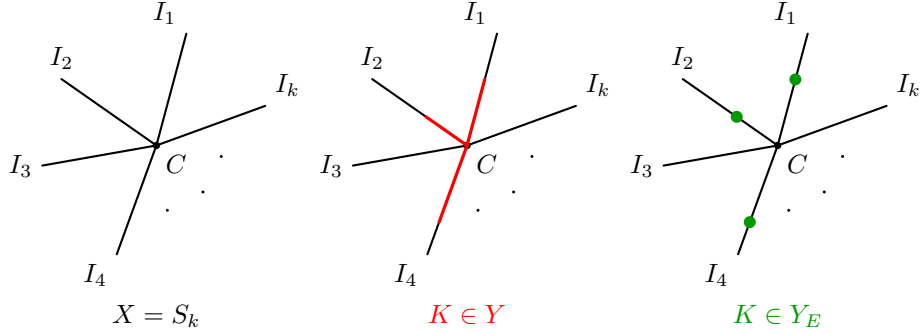
\begin{figure}[ht]
\centering
\begin{tikzpicture}[
  scale=0.70, line cap=round, line join=round,
  brush/.style={thick},
  redseg/.style={very thick, red},
  gdot/.style={circle, fill=green!60!black, inner sep=1.6pt}
]

\def\L{2.2}
\def\xsep{5.9}
\def\xoff{-2.2}
\def\ycap{-2.85}

\def\aA{105}
\def\aB{35}
\def\aC{-10}
\def\aD{-70}
\def\aE{160}

\pgfmathsetmacro{\bA}{180-\aA}
\pgfmathsetmacro{\bB}{180-\aB}
\pgfmathsetmacro{\bC}{180-\aC}
\pgfmathsetmacro{\bD}{180-\aD}
\pgfmathsetmacro{\bE}{180-\aE}

\pgfmathsetmacro{\xone}{\xoff+\xsep}
\pgfmathsetmacro{\xtwo}{\xoff+2*\xsep}

\def\rEll{1.25}
\pgfmathsetmacro{\diffBE}{\bE-\bD-360}

\newcommand{\DrawBrushBase}{%
  
  \fill (0,0) circle (2pt) node[below right] {$C$};

  \draw[brush] (0,0) -- ({\L*cos(\bA)},{\L*sin(\bA)}) node[above left] {$I_1$};
  \draw[brush] (0,0) -- ({\L*cos(\bB)},{\L*sin(\bB)}) node[above]      {$I_2$};
  \draw[brush] (0,0) -- ({\L*cos(\bC)},{\L*sin(\bC)}) node[left]       {$I_3$};
  \draw[brush] (0,0) -- ({\L*cos(\bD)},{\L*sin(\bD)}) node[below left] {$I_4$};
  \draw[brush] (0,0) -- ({\L*cos(\bE)},{\L*sin(\bE)}) node[above right]{$I_k$};

  \foreach \t in {0.44,0.50,0.56}{
    \pgfmathsetmacro{\ang}{\bD + \t*\diffBE}
    \fill ({\rEll*cos(\ang)},{\rEll*sin(\ang)}) circle (0.8pt);
  }%
}

% 1) X = S_k
\begin{scope}[xshift=\xoff cm]
  \DrawBrushBase
  \node[below] at (0,\ycap) {$X=S_k$};
\end{scope}

% 2) K in Y
\begin{scope}[xshift=\xone cm]
  \DrawBrushBase
  \draw[redseg] (0,0) -- ({1.30*cos(\bA)},{1.30*sin(\bA)});
  \draw[redseg] (0,0) -- ({0.95*cos(\bB)},{0.95*sin(\bB)});
  \draw[redseg] (0,0) -- ({1.55*cos(\bD)},{1.55*sin(\bD)});
  \node[below, text=red] at (0,\ycap) {$K\in Y$};
\end{scope}

% 3) K in {Y}_E
\begin{scope}[xshift=\xtwo cm]
  \DrawBrushBase
  \node[gdot] at ({1.30*cos(\bA)},{1.30*sin(\bA)}) {};
  \node[gdot] at ({0.95*cos(\bB)},{0.95*sin(\bB)}) {};
  \node[gdot] at ({1.55*cos(\bD)},{1.55*sin(\bD)}) {};
  \node[below, text=green!60!black] at (0,\ycap) {$K\in {Y}_E$}; % <-- green text
\end{scope}

\end{tikzpicture}
\caption{ $S_k$, a typical element in $Y$, a typical element in ${Y}_E$.}
\end{figure}

\vspace{1mm}
\noindent{\bf Step III.} 
Since 
$$h_{\mathrm{pol}}(C(f))=h_{\mathrm{pol}}(C(f)^m)=h_{\mathrm{pol}}(C(f^m)),$$ we will assume that 
\begin{equation}\label{eq:f(I)=I}
f(I_j)=I_j
\end{equation}
for $j\in\{1,\ldots,k\}$. Denote by $Y$ the set of all stars contained in $X$ with endpoints on the edges of $X$, i.e.,
\begin{equation}\label{eq:Y}
Y:=\left\{K\in C(X)\middle |\,
\begin{aligned}
&E(K)=\{x_{1},\ldots,x_{l}\},\quad\mbox{where}\\
&x_{j}\in I_{n_j},\; n_j\in\{1,\ldots,k\},\quad n_j\neq n_i,\;\mbox{for}\,j\neq i;\\
&l\in\{2,\ldots,k\}
\end{aligned}\right\}
\end{equation}
and by
\begin{equation}\label{eq:tildeY}
{Y}_E:=\{E(K)\mid K\in Y\}\subset F_k(X),
\end{equation} see Figure 1. It is obvious that $Y$ is a closed subset of $C(X)$ and ${Y}_E$ is a closed subset of $F_k(X)$. From~(\ref{eq:f(I)=I}) it follows that they also are $2^f$-invariant. 
Note that $Y$ and ${Y}_E$ are homeomorphic subspaces of $2^X$, since
\begin{equation}\label{eq:psi}
\psi:Y\to{Y}_E,\quad\psi:K\mapsto E(K)
\end{equation} is a homeomorphism with respect to the Hausdorff topology, and that $\psi$ also establishes the conjugacy between $2^f|_Y$ and $2^f|_{Y_E}$. We therefore conclude
$$h_{\mathrm{pol}}(C(f))=h_{\mathrm{pol}}\left(2^f|_{C(X)}\right)\ge h_{\mathrm{pol}}\left(2^f|_{Y}\right)= h_{\mathrm{pol}}\left(2^f|_{Y_E}\right).$$

\vspace{1mm}
\noindent{\bf Step IV.} We prove that $h_{\mathrm{pol}}\left(2^f|_{{Y}_E}\right)\ge k$, and conclude
\begin{equation}\label{eq:ge}
h_{\mathrm{pol}}(C(f))\ge k.
\end{equation}
Denote by $x^j\in I_j$ any wandering point and by $x_n^j:=f^n(x^j)$, for $j\in\{1,\ldots,k\}$, $n\in\mathbb{Z}$. Define the map:
$$\varphi:2^X\to\left(\{0,1\}^k\right)^\mathbb{Z},\quad \varphi(K):=\{u_n^j\},$$ for
$j\in\{1,\ldots,k\}$, $n\in\mathbb{Z}$, where 
$$u_n^j:=\begin{cases} 1, &x_n^j\in K,\\
0, &x_n^j\notin K.\end{cases}$$ Since $x_n^j=f^n(x_j)$, it folllows that
\begin{equation}\label{eq:conjugacy}
\varphi\circ 2^f=\sigma\circ\varphi,
\end{equation}
where 
$$\sigma:\left(\{0,1\}^k\right)^\mathbb{Z}\to\left(\{0,1\}^k\right)^\mathbb{Z}$$ is the shift map defined by
\begin{equation}\label{eq:shift-k}
\sigma(\{b_n^j\}):=\{b_{n+1}^j\},\quad j\in\{1,\ldots,k\},\;n\in\mathbb{Z}.
\end{equation}
The space $\left(\{0,1\}^k\right)^\mathbb{Z}$ can be identified with the space 
$$\{\alpha_1,\ldots,\alpha_{2^k}\}^\mathbb{Z}$$ of two-sided sequences over the alphabet $\{\alpha_1,\ldots,\alpha_{2^k}\}$, equipped with the standard shift map~(\ref{eq:shift-A}), here explicitly defined by~(\ref{eq:shift-k}). We can think of $\alpha_j$ as a $k$-tuple $(s_1,\ldots,s_k)$ where $s_j\in\{0,1\}$.

Denote by $\Sigma:=\varphi({Y}_E)$.  

The set $\Sigma$ consists of all $k$-tuples $(\{v_n^1\},\ldots,\{v_n^k\})$ of sequences such that $v_n^j$ is equal to one for at most one $n\in\mathbb{Z}$, for $j\in\{1,\ldots,k\}$.

We want to show that
\begin{equation}\label{eq:<=}
h_{\mathrm{pol}}\left(2^f;{{Y}_E}\right)\ge h_{\mathrm{pol}}(\sigma;{\Sigma}).
\end{equation}
We cannot simply use the conjugacy~(\ref{eq:conjugacy}) and the fact that $(\sigma;{\Sigma})$ is a factor of $(2^f;{{Y}_E})$ because the map $\varphi$ is not continuous (indeed, to see this take $K_1:=\{a_0^1\}=\{x_1\}$ and $K_2$ to be any singleton $\{x\}$ where $I_1\ni x\neq x^1$ is close enough to $x^1$). However, we will show that~(\ref{eq:<=}) nevertheless holds.

Fix $\varepsilon>0$. Choose $m\in\mathbb{N}$ such that for any two sequences $\{s_n\}, \{t_n\}\in\{0,1\}^\mathbb{Z}$,
$$d(\{s_n\},\{t_n\})<\varepsilon\;\Leftrightarrow\;s_r=t_r\;\mbox{for all}\;r\in[-m,m].$$ Define 
\begin{equation}\label{eq:delta}
\delta:=\min\{d(x^i_r,x^j_l)\mid r,l\in[-m,m],\;i,j\in\{1,\ldots,k\},\;i\neq j\}.
\end{equation}
Since the points $x^j$ are wandering, and belong to different $I_j$'s, $\delta>0$.

Now take $n\in\mathbb{N}$ and let $E=\{\mathbf{a^1},\ldots,\mathbf{a}^N\}\subset\Sigma$ be a $(\sigma,n,\varepsilon)$-separated set in $\Sigma$. Define
$$K_j:=\{x^l_r\mid(\mathbf{a}^j)^l_r=1\}.$$ From the definitions of ${Y}_E$ and $\Sigma$ we see that $K_j$ is indeed an element of ${Y}_E$.
It is easy to see that the set $\{K_1,\ldots,K_N\}$ is $(2^f,n,\delta)$-separated: for $i,j\in\{1,\ldots,N\}$, since $E$ is $(\sigma,n,\varepsilon)$-separated,
there exists $l\in\{0,\ldots,n-1\}$ such that 
$$d(\sigma^l(\mathbf{a}^i),\sigma^l(\mathbf{a}^j))>\varepsilon.$$ 
By the choice of $m$, this means that there exists $r\in[-m,m]$ such that
$$\sigma^l(\mathbf{a}^i)_r\neq\sigma^l(\mathbf{a}^j)_r,$$
or, equivanently, $a_{r+l}^i\neq a_{r+l}^j$. By definitions of $K_i$ and $K_j$, this implies that one of the sets $(2^f)^l(K_i)$, $(2^f)^l(K_j)$ contains the point 
$$x\in\{x^l_r\mid r\in[-m,m],\,l\in\{1,\ldots,k\}\},$$ while the other does not. Since the orbit points $x^l_r$, for $r\in[-m,m]$, $l\in\{1,\ldots,k\}$ are pairwise at least $\delta>0$ apart, by construction, we obtain
$$d\left((2^f)^l(K_i),(2^f)^l(K_j)\right)>\delta.$$
Therefore
$$\mathrm{sep}(2^f,n,\delta;{Y}_E)\ge\mathrm{sep}(\sigma,n,\epsilon;\Sigma).$$ We have that
$$\begin{aligned}
&h(2^f;{Y}_E)=\lim_{\delta\to 0}\limsup_{n\to\infty}\frac{\log\mathrm{sep}(2^f,n,\delta;{Y}_E)}{n}\ge\\
&\limsup_{n\to\infty}\frac{\log\mathrm{sep}(2^f,n,\delta;{Y}_E)}{n}\ge
\limsup_{n\to\infty}\frac{\log\mathrm{sep}(\sigma,n,\varepsilon;\Sigma)}{n}
\end{aligned}$$ and this is true for every $\varepsilon\ge 0$. Therefore we have
$$h(2^f;{Y}_E)\ge h(\sigma;\Sigma).$$

Finally, in order to apply Theorem~\ref{thm:complexity}, we need to compute the complexity function $p_\Sigma(n)$.
 
Take a point 
$$K=\{y^1,\ldots,y^k\}\in{Y}_E.$$ Then $a_n^j\in K$ if and only if $a_n^j=y^j$. Therefore 
$$\varphi(K)=u_n^j=\begin{cases}0, &f^n(x^j)\neq y^j,\\1, &f^n(x^j)= y^j.\end{cases}$$ Since the point $x^j$ is wandering, $u_n^j$ is equal to $1$ at most once, for $j\in\{1,\ldots,k\}$. Hence $p_\Sigma(n)=n^k$ and, therefore
$$h_{\mathrm{pol}}(C(f))\ge\lim_{n\to\infty}\frac{\log p_\Sigma(n)}{\log n}= k.$$

\vspace{1mm}
\noindent{\bf Step V.} We prove that 
\begin{equation}\label{eq:le}
h_{\mathrm{pol}}(C(f))\le k.
\end{equation}
We still assume that $f(I_j)=I_j$.
The set $C(X)$ can be divided into two closed and $C(f)$-invariant subsets, namely
$$C(X)=Y\cup \overline{C(X)\setminus Y},$$
where $Y$ is defined in~(\ref{eq:Y}). The elements of $\overline{C(X)\setminus Y}$ are precisely the intervals whose endpoints lie in the same edge, including singletons. It holds that
$$Y\cap \overline{C(X)\setminus Y}=\{\{b\}\}.$$
We have
$$h_{\mathrm{pol}}(C(f))=\max\left\{h_{\mathrm{pol}}(C(f)|_Y),h_{\mathrm{pol}}\left(C(f)|_{\overline{C(X)\setminus Y}}\right)\right\},$$
by~(\ref{eq:finite-union}). It holds that
$$h_{\mathrm{pol}}\left(C(f)|_{\overline{C(X)\setminus Y}}\right)=2.$$ 
Indeed
$$h_{\mathrm{pol}}\left(C(f)|_{\overline{C(X)\setminus Y}}\right)=\max_{1\le j\le k}\{h_{\mathrm{pol}}(C(f|_{I_J}))\}$$ and
$$h_{\mathrm{pol}}(C(f|_{I_J}))=2$$
(see~\cite{DK}). Therefore, it remains to show that
$$f_{\mathrm{pol}}(C(f)|_Y)\le k.$$

%$$C(X)=Y\cup\bigcup_{j=1}^kY_j,$$ where 
%$$Y_j:=\{K\in C(f)\mid K\subset I_j\}.$$ Therefore
%$$h_{\mathrm{pol}}(C(f))\le\max\{h_{\mathrm{pol}}(C(f))|_{Y},h_{\mathrm{pol}}(C(f))|_{Y_1},\ldots,h_{\mathrm{pol}}(C(f))|_{Y_k}.$$
%We know from~\cite{DK} that
%$$h_{\mathrm{pol}}(C(f))|_{Y_1}=\ldots,h_{pol=}(C(f))|_{Y_k}=2.$$
We can construct a continuous surjections
$$\pi:I_1\times\ldots\times I_k\to {Y}_E$$ 
(where ${Y}_E$ is defined in~(\ref{eq:Y})) as:
$$\pi:(x_1,\ldots,x_k)\mapsto\{x_1,\ldots,x_k\}.$$ 
It is obvious that 
$$\pi\circ (f|_{I_1}\times\ldots\times f|_{I_k})=2^f\circ\pi,$$ 
so
$$h_{\mathrm{pol}}(2^f|_{{Y}_E})\le h_{\mathrm{pol}}(f|_{I_1}\times\ldots\times f|_{I_k}).$$
We already know from~(\ref{eq:psi}) that
$$h_{\mathrm{pol}}(2^f|_Y)=h_{\mathrm{pol}}(2^f|_{{Y}_E})$$
so
$$h_{\mathrm{pol}}(C(f)|_Y)=h_{\mathrm{pol}}(2^f|_Y)\le h_{\mathrm{pol}}(f|_{I_1}\times\ldots\times f|_{I_k})=h_{\mathrm{pol}}(f|_{I_1})+\ldots+h_{\mathrm{pol}}(f|_{I_k}).$$
It remains to note that $h_{\mathrm{pol}}(f|_{I_j})=1$, which is precisely the content of Lemma 3.1 in~\cite{L}.\qed

\begin{rem}
Note that in the proof of the previous theorem we used the fact that $x^j$ is wandering only to prove $\delta>0$ in~(\ref{eq:delta}) and in the computation of $p_\Sigma$. The same result could be obtained under the assumption that $f^m(x^j)\neq f^m(x^j)$ for all $m,n\in\mathbb{Z}$, $m\neq n$ and all $j\in\{1,\ldots,k\}$. However, these two assumtion are equivalent in the case of a homeomoprhism of an interval, which $f^m|_{I_j}$ is.
\end{rem}

\medskip

We have the following two corollaries.

\begin{cor}
Let $X$ be as in Theorem~\ref{thrm:star} and $f:X\to X$ a homeomorphism such that there are exactly $l\le k$ edges $I_j$ that possess a wandering point. Then
$h_{\mathrm{pol}}(C(f))=\max\{l,2\}$ if $l\ge 1$, and $h_{\mathrm{pol}}(C(f))=0$ if $l=0$.
\end{cor}

\noindent{\it Proof.} As in the proof of Theorem~\ref{thrm:star}, take $m\in\mathbb{N}$ such that $f^m:I_j\to I_j$. Let $I_1,\ldots,I_l$ be the edges with a wandering point and $I_{l+1},\ldots,I_k$ the edges without a wandering point. Then
$$f^m|_{I_j}=\mathrm{Id},\quad\mbox{for}\;j\in\{l+1,\ldots,k\}.$$ Denote by
$$A:=\bigcup_{j=1}^lI_j,\quad, B:=\bigcup_{j=l+1}^kI_j.$$
We have
$$h_{\mathrm{pol}}(C(f))=\max\{h_{\mathrm{pol}}(C(f)|_A),h_{\mathrm{pol}}(C(f)|_B)\}=h_{\mathrm{pol}}(C(f)|_A).$$
The proof now follows from Theorem~\ref{thrm:star}.\qed

Corollary~\ref{cor:graph1} provides a complete description of $h_{\mathrm{pol}}(C(f))$ for a homeomorphism $f$ of a graph. We restate it for the sake of convenience.

\begin{cor} Let $G$ be a graph and $f:G\to G$ a homeomorphism. Let $k$ be the maximal number of edges with a common branch point that possess a wandering point. Then
\begin{itemize}
\item $h_{\mathrm{pol}}(C(f))=k$ if $k\ge 2$;
\item $h_{\mathrm{pol}}(C(f))=2$ if $k\le 1$ and and there exists a closed path $C\subset X$ such that $f^n|_{C}$ is not conjugate to the circle rotation, for all $n$;
\item $h_{\mathrm{pol}}(C(f))=0$ otherwise.
\end{itemize}
\end{cor}

\noindent{\it Proof.} Let $b_1,\ldots,b_r$ denote all branch points in $X$ and $C_1,\ldots,C_p$ denote all closed path contained in $X$. Since $f$ is a homeomorphism, every branch point is mapped to a branch point (of the same order) and $f(C_i)=C_j$ for some $j$. Therefore, there exists $m\in\mathbb{N}$ such that 
$$f^m:b_j\mapsto b_j,\quad f^m|_{C_j}:C_j\to C_j.$$ We have
$$h_{\mathrm{pol}}(C(f))=h_{\mathrm{pol}}(C(f^m))=\max_{i,j}\{h_{\mathrm{pol}}(C(f^m)|_{S_j}), h_{\mathrm{pol}}(C(f^m)|_{C_i})\}$$ where $S_j$ denotes the star with a branch point $b_j$. We now apply Theorem~\ref{thrm:star} and the result from~\cite{DjK}, which says that if a homeomorphism $f:\mathbb{S}^1\to\mathbb{S}^1$ of the circle is not conjugated to a rotation, then $h_{\mathrm{pol}}(C(f))=2$.  \qed

\section{Topological entropy of $C(f)$ on manifolds}

Let us recall the second main result announced in the Introduction.

\begin{thrm}
Let $f:M\to M$ be a reversible dynamical system on a connected compact topological manifold $M$ with the dimension at least two. If $f$ possesses a wandering point, then $h(C(f))=\infty$.
\end{thrm}

\noindent{\it Proof.} Let $x_0$ be wandering and $U\ni x_0$ an open set such that $f^n(U)\cap f^m(U)=\emptyset$ for all integers $m$ and $n$. Fix a positive integer $k$ and choose $k$ different points $x^1,\ldots,x^k\in U$. %Since every sequence $f^n(x_j)$ has a convergent subequence (when $n\to\pm\infty$), we can assume that
%$$\lim_{n\to-\infty}f^n(x_j)=a_j,\quad\lim_{n\to+\infty}f^n(x_j)=b_j,\quad j\in\{1,\ldots,k\}.$$
Define the map
$$\varphi:2^M\to \left(\{0,1\}^k\right)^\mathbb{Z}$$ as
$$\varphi:K\mapsto\{v^j_n\},\; j\in\{1,\ldots,k\},\,n\in\mathbb{Z}$$
where
$$v^j_n:=\begin{cases}
0,&x^j_n\notin K\\
1,&x^j_n\in K,
\end{cases}$$
where $x^j_n:=f^n(x^j)$.
We again have $\varphi\circ 2^f=\sigma\circ\varphi$, where 
$\sigma$ is the shift map~(\ref{eq:shift-k}).

However, as mentioned in the proof of Theorem~\ref{thrm:star}, $\varphi$ is not continuous so we cannot use a standard factor argument.

Consider $F_\omega\subset 2^M$ and denote by $\Sigma:=\varphi(F_\omega)$. Let us prove that
$$
h(2^f;F_\omega)\ge h(\sigma;\Sigma)
$$
is still true. Similarly to the proof of Theorem~\ref{thrm:star}, fix $\varepsilon>0$ and choose $m\in\mathbb{N}$ such that for any two sequences $\{s_n\}, \{t_n\}\in\{0,1\}^\mathbb{Z}$,
$$d(\{s_n\},\{t_n\})<\varepsilon\;\Leftrightarrow\;s_r=t_r\;\mbox{for all}\;r\in[-m,m].$$ Define 
$$\delta:=\min\{d(x^i_r,x^j_l)\mid r,l\in[-m,m]\;,i,j\in\{1,\ldots,k\},\;(i,r)\neq (j,l)\}.$$ Since all $x^j$'s are wandering, we have $\delta>0$. Indeed, $x^i_r\neq x^i_l$ since $x^i$ is not periodic. Suppose that $x^i_r=x^j_l$ for some $i\neq j$. If $r=l$, then $x^i=x^j$, which implies $i=j$, a contradiction. If $r>l$, we have $f^s(x^i)=x^j$, for $s:=r-l>0$, which contradicts the fact that $U$ is a wandering neighbourhood.

Again, for $n\in\mathbb{N}$, let $E=\{\mathbf{a^1},\ldots,\mathbf{a}^N\}\subset\Sigma$ be a $(\sigma,n,\varepsilon)$-separated set in $\Sigma$. Define
$$K_j:=\{x^l_r\mid(\mathbf{a}^j)^l_r=1\}\in F_\omega.$$
By the same argument as in the proof of Theorem~\ref{thrm:star}, we obtain that the set $\{K_1,\ldots,K_N\}$ is $(2^f,n,\delta)$-separated. 

Therefore 
$$\begin{aligned}
&h(2^f;Y)=\lim_{\delta\to 0}\limsup_{n\to\infty}\frac{\log\mathrm{sep}(2^f,n,\delta;Y)}{n}\ge\\
&\limsup_{n\to\infty}\frac{\log\mathrm{sep}(2^f,n,\delta;Y)}{n}\ge
\limsup_{n\to\infty}\frac{\log\mathrm{sep}(\sigma,n,\varepsilon;\Sigma)}{n}.
\end{aligned}$$ Since this equality holds for every $\varepsilon\ge 0$, we conclude that
$$h(2^f;F_\omega)\ge h(\sigma;\Sigma).$$

Let us now compute $h(\sigma;\Sigma)$.
It is obvious that any word of the lenght $m$ can be obtained as a part of a sequence in $\Sigma$. Therefore $p_\Sigma(m)=(2^k)^m$ and
$$\lim_{m\to\infty}\frac{\log p_\Sigma(m)}{m}=\log 2^k,$$ hence, by Theorem~\ref{thm:complexity}, $h(\sigma;\Sigma)=2^k$. %Since $h(C(f))\ge \log 2^k$ for any $k$, we have $h(C(f))=\infty$.

It remains to estimate the entropy of $C(f)$.
Fix $k$ and let $Y_k\subset C(M)$ be the set of all $K$ that are homeomorphic image of a star with the order $k$. 
Define
$$Y:=\bigcup_{k\in N}Y_k\subset C(M).$$ The set $Y$ is $C(f)$-invariant but not closed subset of $2^M$. The same holds for $F_\omega$. Consider the following contiuous map
$$\pi:Y\to F_\omega,\quad \pi:K\mapsto\mathrm E(K).$$ This map is Lipschitz with a Lipschitz constant $1$. It is also surjective, since $M$ is a connected (hence path-connected) manifold (and therefore locally Euclidean). Hence we can apply Proposition~\ref{prop:lip} and conclude
$$h(C(f))\ge h(C(f),Y)\ge h(2^f;F_\omega)\ge 2^k.$$ Since the last inequality holds for all $k$, the claim follows. \qed

\section{Proof of Theorem~\ref{thm:complexity}}\label{sec:proof}

In order to prove Theorem~\ref{thm:complexity}, we first need to establish the notion of topological and polynomial entropy relative to $Y$ via open covers. %The construction and proofs are completely analogous for both types of entropy, so we only give detailed proofs for the topological entropy.

A family $\mathcal{U}$ of open subsets of $X$ is an {\it open cover} if $X=\bigcup_{U\in\mathcal{U}}U$. 
For an open cover $\mathcal{U}$ of $X$ and $Y\subseteq X$, we define
$$\mathcal{U}(Y):=\{U\in\mathcal{U}\mid U\cap Y\neq\emptyset\}.$$

We say that the open cover $\mathcal{V}$ is {\it finer} than the open cover $\mathcal{U}$ if for ever $V\in\mathcal{V}$ there exists $U\in\mathcal{U}$ with $V\subset U$, and we denote this by $\mathcal{V}\succeq\mathcal{U}$.
It is obvious that
$$\mathcal{V}\succeq\mathcal{U}\;\Rightarrow\mathcal{V}(Y)\succeq\mathcal{U}(Y).$$

We deal only with finite open covers (since $X$ is compact this is not a real restriction). The {\it diameter} of a finite open cover $\mathcal{U}$ is defined as:
$$\mathrm{diam}(\mathcal{U}):=\max\{\mathrm{diam}(U)\mid U\in\mathcal{U}\}.$$

Let $\mathcal{U}_1(Y),\ldots,\mathcal{U}_m(Y)$ be $m$ collections of open sets. Define their {\it join} by:
$$\bigvee_{j=1}^m\mathcal{U}_j:=\left\{\bigcap_{j=1}^mU_j\mid U_j\in\mathcal{U}_j\right\}.$$ 
If each $\mathcal{U}_j$ is a cover of $X$, then $\bigvee_{j=1}^m\mathcal{U}_j$ is also a cover of $X$. It is not hard to see that 
$$\left(\bigvee_{j=1}^m\mathcal{U}_j\right)(Y)\subseteq\bigvee_{j=1}^m\mathcal{U}_j(Y)$$ but that the equality does not have to hold.

If $f:X\to X$ is a continuous map, $\mathcal{U}$ a collection of open sets, and $m\ge 0$ is an integer, define:
$$f^{-m}\mathcal{U}:=\left\{f^{-m}(U)\mid U\in\mathcal{U}\right\}.$$ If $\mathcal{U}$ is a cover of $X$, then $f^{-m}\mathcal{U}$ is also a cover of $X$. If $Y\subseteq X$ is $f$-invariant (i.e., $f(Y)=Y)$) and $\mathcal{U}$ is a cover of $X$, then 
$$f^{-m}\left(\mathcal{U}(Y)\right)=(f^{-m}\mathcal{U})(Y).$$
For an open cover $\mathcal{U}$ of $X$, denote by
$$\mathcal{U}^n:=\bigvee_{j=0}^{n-1}f^{-j}(\mathcal{U}),\quad \mathcal{U}^n(Y):=\left(\bigvee_{j=0}^{n-1}f^{-j}(\mathcal{U})\right)(Y).$$ For a finite open cover $\mathcal{U}$ of $X$, we denote by $N(\mathcal{U}(Y))$ the minimal cardinality among all finite subcovers of $\mathcal{U}(Y)$ that still cover $Y$.

Finally, the {\it topological entropy of} $f$ {\it relative to the open cover} $\mathcal{U}$ and $Y$ is defined as:
$$h(f,\mathcal{U};Y):=\limsup_{n\to\infty}\frac{\log N\left(\mathcal{U}^n(Y)\right)}{n},$$ and, similarly, we define the {\it polynomial entropy of} $f$ {\it relative to the open cover} $\mathcal{U}$ and $Y$ as:
$$h_{\mathrm{pol}}(f,\mathcal{U};Y):=\limsup_{n\to\infty}\frac{\log N\left(\mathcal{U}^n(Y)\right)}{\log n}.$$

 It is easy to see that for
$$\begin{aligned}
\mathcal{V}\succeq\mathcal{U}\;&\Rightarrow\;N(\mathcal{V}(Y))\ge N(\mathcal{U}(Y))\;\Rightarrow\;
N\left(\mathcal{V}^n(Y)\right)
\ge N\left(\mathcal{U}^n(Y)\right)\\
&\Rightarrow h(f,\mathcal{V};Y)\ge h(f,\mathcal{U};Y).
\end{aligned}$$

Define 
$$h^*(f;Y)=\sup_\mathcal{U} h(f,\mathcal{U};Y),\quad h^*_{\mathrm{pol}}(f;Y)=\sup_\mathcal{U} h_{\mathrm{pol}}(f,\mathcal{U};Y).$$ 

\begin{prop}\label{prop:generating} Let $\mathcal{V}_n$ be a sequence of open covers such that $\lim\limits_{n\to\infty}\mathrm{diam}\left(\mathcal{V}_n\right)\to 0$ and $Y\subseteq X$ an $f$-invariant subset. Then the limit $\lim\limits_{n\to\infty}h\left(f,\mathcal{V}_n(Y)\right)$ exists and
$$h^*(f;Y)=\lim_{n\to\infty}h\left(f,\mathcal{V}_n(Y)\right).$$
The same holds for the polynomial entropy.
\end{prop}

\noindent{\it Proof.} Fix an open cover $\mathcal{U}$ and let $\delta>0$ be its Lebesgue covering number. Choose $n\ge n_0$ such that $\mathrm{diam}\left(\mathcal{V}_n\right)<\delta$, for $n\ge n_0$. Since $\mathcal{V}_n\succeq\mathcal{U}$ we have $\mathcal{V}_n(Y)\succeq\mathcal{U}(Y)$ and therefore $h\left(f,\mathcal{V}_n;Y\right)\ge h(f,\mathcal{U};Y)$. Thus $\liminf\limits_{n\to\infty}h\left(f,\mathcal{V}_n;Y\right)\ge h(f,\mathcal{U};Y)$, and, since $\mathcal{U}$ was arbitrary, we conclude
$$\liminf_{n\to\infty}h\left(f,\mathcal{V}_n\;Y\right)\ge \sup_{\mathcal{U}}\{h(f,\mathcal{U};Y)\}\ge\limsup_{n\to\infty}h\left(f,\mathcal{V}_n;Y\right).$$

\qed 

\begin{thrm} For $Y\subset X$ such that $f(Y)=Y$, we have $h^*(f;Y)=h(f;Y)$ and $h*_{\mathrm{pol}}(f;Y)=h_{\mathrm{pol}}(f;Y)$.
\end{thrm}

\noindent{\it Proof.} We have already defined $\mathrm{sep}(n,\varepsilon;Y)$ on page~\pageref{def:sep}. A subset $F\subset X$ is said to be an {\it $(n,\varepsilon)$-spanning set} for $Y$ (with respect to $f$) if for every $y\in Y$ there exists $x\in F$ such that $d_n^f(x,y)\le\varepsilon$, or, equivalently:
$$Y\subseteq\bigcup_{x\in F}\bigcap_{j=0}^{n-1}f^j\left(B[f^j(x),\varepsilon]\right).$$ Denote by $\mathrm{span}(n,\varepsilon;Y)$ the minimal cardinality of $(f,n,\varepsilon)$-spanning set for $Y$. Since every $(n,\varepsilon)$-separated subset of $Y$ of maximal cardinality is also $(n,\varepsilon)$-spanning set for $Y$, we have 
\begin{equation}\label{eq:span<sep}
\mathrm{span}(n,\varepsilon;Y)\le\mathrm{sep}(n,\varepsilon;Y).
\end{equation}

We now prove the following two inequalities:
\begin{itemize}
\item If $\mathcal{U}$ is an open cover of $X$ with Lebesgue number $\delta$, then
\begin{equation}\label{eq:<1}
N(\mathcal{U}^n(Y))\le\mathrm{span}(n,\delta/2;Y).
\end{equation}

\item If the diameter of an open cover $\mathcal{V}$ is less than or equal to $\varepsilon$, then
\begin{equation}\label{eq:<2}
\mathrm{sep}(n,\varepsilon;Y)\le N(\mathcal{V}^n(Y)).
\end{equation}
\end{itemize}
Let $F$ be an $(n,\delta/2)$-spanning set for Y of cardinality $\mathrm{span}(n,\delta/2,Y)$. This means that
$$Y\subseteq\bigcup_{x\in F}\bigcap_{j=0}^{n-1}f^{-j}(B[f^j(x),\delta/2]).$$ By the choice of $\delta$, we see that
for every $j$, there exists $U_j\in\mathcal{U}$ such that
$$B[f^j(x),\delta/2]\subset U_j.$$ This implies that for every $x\in F$,
$$\begin{aligned}
&B[(x),\delta/2]\cap f^{-1}(B[f(x),\delta/2])\cap\ldots\cap f^{-(n-1)}(B[f^{n-1}(x),\delta/2])\subseteq\\
&U_0\cap f^{-1}(U_1)\cap\ldots\cap f^{-(n-1)}(U_{n-1})
\end{aligned}$$ for some $U_0,U_1,\ldots,U_{n-1}\in\mathcal{U}$. Thus for every $x\in F$, we have a set $U_x\in\mathcal{U}^n$, and 
$$Y\subseteq \bigcup_{x\in F} U_x.$$ Hence the minimal cardinality of subcovers of $\mathcal{U}^n(Y)$ is at most the cardinality of $F$, and~(\ref{eq:<1}) is proved.

Regarding the inequality~(\ref{eq:<2}), suppose that $E\subset Y$ is an $(n,\varepsilon)$ separated set of maximal cardinality $(n,\varepsilon;Y)$. If $x,y\in E$, then $x$ and $y$ cannot belong to the same element of $\mathcal{V}^n(Y)$. Indeed, if $x,y\in U$, $U\in\mathcal{V}^n(Y)$, 
we have
$$x,y\in U=U_0\cap f^{-1}(U_1)\cap\ldots\cap f^{-(n-1)}(U_{n-1}),$$ so $f^j(x), f^j(y)\in U_j$, and since
 the diameter of $U_j$ is less than $\varepsilon$, we would have $d_n^f(x,y)<\varepsilon$, which is a contradiction. Hence, no two points of $E$ can belong to the same element of $\mathcal{V}^n(Y)$, so the minimal number of elements of $\mathcal{V}^n(Y)$ needed to cover $Y$ is at least $\mathrm{sep}(n,\varepsilon; Y)$, and~(\ref{eq:<2}) is proved.

We now complete the proof of the theorem. Let $\varepsilon>0$. Let $\mathcal{U}_\varepsilon$ be any open cover of $X$ consisting of open balls of radius $2\varepsilon$ and $\mathcal{V}_\varepsilon$ be any cover of $X$ consisting of open balls of radius $\varepsilon/2$. From~(\ref{eq:span<sep}),~(\ref{eq:<1}) and~(\ref{eq:<2}), we have:
$$N(\mathcal{U}^n_\varepsilon(Y))\le\mathrm{span}(n,\varepsilon;Y)\le\mathrm{sep}(n,\varepsilon;Y)
\le N(\mathcal{V}^n_\varepsilon(Y)).$$
If we take $\varepsilon=1/n$, then by Proposition~\ref{prop:generating} the outer limits exists and are equal to $h^*(f;Y)$ Hence the inner ones exist too, and $h(f;Y)=h^*(f;Y)$. \qed

%N\left(\mathcal{U}^n(Y)\right)

\vspace{5mm}
\noindent{\bf Proof of Theorem~\ref{thm:complexity}.} The proof is now standard. 
Let $j_i\in\mathbb{Z}$ (for $i\in\{1,\ldots,m\}$) be $m$ distinct integers and let $a_{j_i}\in A$ be $m$ (possibly equal) symbols. Denote by $[a_{j_1}\cdots a_{j_m}]$ the cylinder:
$$[a_{j_1}\cdots a_{j_m}]:=\left\{x\in\Sigma\mid x_{j_i}=a_{j_i},\;i\in\{1,\ldots,m\} \right\}.$$ Each cylinder is both open and closed set in $\Sigma$. For a fixed $n\ge 0$, the family:
$$\mathcal{V}_n:=\left\{[a_{-n}a_{-n+1}\cdots a_0\cdots a_n]\mid a_j\in A,\;j\in\{-n,\ldots,n\} \right\}$$ forms an open cover of $A^\mathbb{Z}$ with diameter smaller than $2^{-n}$. From Proposition~\ref{prop:generating} it follows that
$$h(\sigma;\Sigma)=\lim_{n\to\infty}h(f,\mathcal{V}_n;\Sigma).$$
 Since $\sigma^{-k}([a_{-n}\cdots a_{n}])=[b_{k-n}\cdots b_{k+n}]$, with $b_l=a_{l-k}$, we conclude
$$\bigvee_{j=0}^{l-1}\sigma^{-j}(\mathcal{V}_n)=\left\{[a_{-n}\cdots a_nb_{n+1}\cdots b_{n+l}]\mid a_j, b_i\in A \right\}.$$
We see that 
$$N\left(\bigvee_{j=0}^{l-1}\sigma^{-j}(\mathcal{V}_n(\Sigma))\right)=p_\Sigma(2n+1+l),$$ so we have
$$h(\sigma,\mathcal{V}_n;\Sigma)=\limsup_{l\to\infty}\frac{\log p_\Sigma(2n+1+l)}{l}=\limsup_{l\to\infty}\frac{\log p_\Sigma(l)}{l}.$$
The proof for the polynomial entropy is the same.\qed

\section{Proof of Theorem~\ref{thm:question}}

We conclude this paper by answering a question posed in~\cite{AB}. Recall that $C_n(X)$ is the closed subset of $2^X$ consisting of all $K\in 2^X$ such that $K$ has at most $n$ connected components.

The authors of~\cite{AB} asked the following question: does there exist a dynamical system $(X,f)$ such that
$$h(C(f))<h(C_n(F))<h(2^f)$$ for some $n\ge 2$? The answer is affirmative. Indeed, consider any interval map $f:[0,1]\to [0,1]$ with $h(f)=\alpha>0$. If we denote $X:=[0,1]$, then we have the following uniformly finite-to-one continuous surjections:

$$\pi:C(X)\to F_2(X),\quad \pi_n: C_n(X)\to F_{2n}(X),\quad K\mapsto \partial(K)$$ and the corresponding restriction of $2^f$ are semi-conjugated, more precisely, $F_2(f)$ is a factor of $C(f)$ and $F_{2n}(X)$ is a factor of $C_n(f)$. In the case $n=1$, the map $\pi$ is in fact a homeomorphism.

Moreover, there are natural projections $\tau_n :X^n\to F_n(X)$ which are uniformly finite-to-one continuous semiconjugacies. It is know that in this case, the topological entropy is preserved (see Theorem 17 in~\cite{Bo}). Therefore,
$$h(C_n(f))=h(F_n(f))=h(f^{\times n})=n\alpha.$$ Consequently,
$$h(f)<f(C(f))<h(C_2(f))<\ldots<h(C_n(f))<h(C_{n+1}(f))<\ldots<h(2^f)=\infty.$$

One may ask the same question for polynomial entropy. However, the previous argument does not apply in this setting, as polynomial entropy is not preserved under uniformly finite-to-one semiconjugacies.

However, we can use (almost) the same example to construct the map satisfying
$$h_{\mathrm{pol}}(C(f))<h_{\mathrm{pol}}(C_n(f))<h_{\mathrm{pol}}(2^f).$$
Let $X:=[0,1]$ and $f:X\to X$ be a homeomorphism with a wandering point (for example $f(x)=x^2$). 
It is known (see~\cite{L}) that $h_{\mathrm{pol}}(f)=1$. We claim that $h_{\mathrm{pol}}(C_n(f))=2n$. It is shown in~\cite{DK} that 
$h_{\mathrm{pol}}(C(f))=2$. To illustrate the general case, we show that $h_{\mathrm{pol}}(C_2(f))=4$. We divide the set $C_2(X)$ 
into four parts and establish the conjugation between every part with some space $F_k(X)$. Then we use that 
\begin{equation}\label{eq:F_k}
h_{\mathrm{pol}}(F_k(f))=k
\end{equation}
(see~\cite{DKL}).
Denote by 
$$\begin{aligned}
&A:=\{K\in C_2(X)\mid K=[a,b]\cup[c,d]\;\mbox{where}\;a<b<c<d\},\\
 &B_1:=\{K\in C_2(X)\mid K=[a,b]\cup\{c\},\;\mbox{where}\;a<b<c\}\\
 &B_2:=\{K\in C_2(X)\mid K=\{a\}\cup[b,c],\;\mbox{where}\;a<b<c\}\\
 &C:=\{K\in C_2(X)\mid K=[a,b],\;\mbox{where}\;a\le b\}.
   \end{aligned}
 $$
 We observe that $C_2(X)=A\sqcup B_1\sqcup B_2\sqcup C$ and that the subsets $A, B_1, B_2$ and $C$ are $C_f(f)$-invariant (but not closed, except for $C$), since $f$ is a homeomorphism. Therefore
 $$h_{\mathrm{pol}}(C_2(f))=\max\{h_{\mathrm{pol}}(C_2(f);A), h_{\mathrm{pol}}(C_2(f);B_1), h_{\mathrm{pol}}(C_2(f);{B_2}), h_{\mathrm{pol}}(C_2(f);C)\}.$$
 Consider the map
 $$\pi:C_2(X)\to F_4(X),\quad, \pi:K\mapsto \partial(K).$$ 
It is clear that $\pi$ is a continuous surjection satisfying $\pi\circ C_2(f)=F_4(f)\circ\pi$. 
Next, define the following partition of $F_4(X)$:
 $$F_4(X)=D\sqcup E\sqcup F$$ as
 $$D:=\{K\in F_4(X)\mid |K|=4\},\quad E:=\{K\in F_4(X)\mid |K|=3\},\quad F=\{K\in F_4(X)\mid |K|\le 2\}.$$
 Note that all the subset are $F_4(f)$-invariant (but not closed) and $E\cup F=F_3(X)$.
 We observe that
 $$\pi|_{A}:A\to D,\quad\pi|_{B_1}:B_1\to E,\quad \pi|_{B_2}:B_2\to E,\quad \pi|_{C}:C\to F$$ are homeomorphisms. Moreover, they are isometries, and hence bi-Lipschitz. Therefore, by Proposition~\ref{prop:lip}
 $$\begin{aligned}
& h_{\mathrm{pol}}(C_2(f);A)=h_{\mathrm{pol}}(F_4(f);D),\\
&h_{\mathrm{pol}}(C_2(f);{B_1})=h_{\mathrm{pol}}(C_2(f);{B_2})=h_{\mathrm{pol}}(F_4(f);E),\\ 
&h_{\mathrm{pol}}(C_2(f);C)=h_{\mathrm{pol}}(F_4(f);F)=2,
\end{aligned}$$ so
$$h_{\mathrm{pol}}(C_2(f))=\max\{h_{\mathrm{pol}}(F_4(f);D),h_{\mathrm{pol}}(F_4(f);E),h_{\mathrm{pol}}(F_4(f);F)\}.$$
We will show that the later is equal to $h_{\mathrm{pol}}(F_4(f);D)=4$. 
First, we note that $h_{\mathrm{pol}}(C_2(f))\le 4$, since $C_2(f)$ is a factor of $F_4(f)$. Therefore, it suffices to prove that $h_{\mathrm{pol}}(F_4(f);D)=4$.
We have
$$\begin{aligned}
4&=h_{\mathrm{pol}}(F_4(f))=\max\{h_{\mathrm{pol}}(F_4(f);D),h_{\mathrm{pol}}(F_4(f);{E\cup F})\}\\
&=\max\{h_{\mathrm{pol}}(F_4(f);D),h_{\mathrm{pol}}(F_4(f);{F_3(X)})\}=\max\{h_{\mathrm{pol}}(F_4(f);D),h_{\mathrm{pol}}(F_3(f))\}\\
&=\max\{h_{\mathrm{pol}}(F_4(f);D),3\},\end{aligned}$$ (the last equality follows from~(\ref{eq:F_k})) so $h_{\mathrm{pol}}(F_4(f);D)=4$.\qed

 \end{document}